\documentclass[reqno,oneside,11pt]{amsart}
\usepackage[T1]{fontenc}
\usepackage[utf8]{inputenc}
\usepackage{lmodern}
\usepackage[french,english]{babel}
\usepackage{geometry} 

\usepackage{amsmath,amssymb,amsfonts,amsthm,mathtools}
\usepackage{mathrsfs}
\usepackage{graphicx}
\usepackage{caption}
\usepackage{subcaption} 
\usepackage{booktabs}
\usepackage{array}
\usepackage{paralist}
\usepackage{fancyhdr}
\usepackage{emptypage}
\fancypagestyle{plain}{
	\fancyhf{}
	\fancyfoot[RO,RE]{\thepage}
	
	}

\usepackage{comment} 
\usepackage{diagbox}
\usepackage{xcolor}
\usepackage{epigraph}

\usepackage[noadjust]{cite}
\usepackage[
pagebackref=true,
colorlinks=true,
urlcolor=purple,
linkcolor=purple!87!black,
citecolor=green!60!black,
pdfborder={0 0 0}
]{hyperref}
\renewcommand*{\backref}[1]{}
\renewcommand*{\backrefalt}[4]{[{\tiny%
		\ifcase #1 Not cited.%
		\or Cited on page~#2.%
		\else Cited on pages #2.%
		\fi%
	}]}
\usepackage{cleveref}

\usepackage{bookmark}
\usepackage{dsfont}

\newcommand{\N}{\mathbb{N}}

\renewcommand{\P}{\mathbb{P}}

\newcommand{\thistheoremname}{}

\newtheorem*{genericthm*}{\thistheoremname}
\newenvironment{namedthm*}[1]
{\renewcommand{\thistheoremname}{#1}%
	\begin{genericthm*}}
	{\end{genericthm*}}
\theoremstyle{plain}
\newtheorem{thm}{Theorem}[section] 

\newtheorem{prop}[thm]{Proposition}
\newtheorem{lem}[thm]{Lemma}
\newtheorem{cor}[thm]{Corollary}
\newtheorem{question}[thm]{Question}
\theoremstyle{definition}

\newtheorem{claim}[thm]{Claim}
\newtheorem{defn}[thm]{Definition} 

\usepackage{bm,thmtools}

\newcommand{\cost}{\mathrm{cost}}

\newcommand{\llangle}{\langle\!\langle}
\newcommand{\rrangle}{\rangle\!\rangle}

\allowdisplaybreaks
\author{Miguel Donoso-Echenique}
\email[Miguel Donoso-Echenique]{mdonosoe@uni-muenster.de, migueldonosoe@gmail.com}
\address{University of Münster, Einsteinstrasse 62, Münster 48149, Germany}

\urladdr{\url{https://sites.google.com/view/miguel-donoso-echenique}}

\author{Eduardo Silva} 
\email[Eduardo Silva]{eduardo.silva@uni-muenster.de, edosilvamuller@gmail.com}
\address{University of Münster, Einsteinstrasse 62, Münster 48149, Germany}

\urladdr{\url{https://edoasd.github.io/eduardo_silva_math/}}
\title[]{Free Burnside groups of large odd exponent have cost 1}
\allowdisplaybreaks
\usepackage{url}

\usepackage{enumitem}
\setlist[enumerate, 1]{label = (\roman*), ref = \roman*}
\setlist[enumerate, 2]{label = \theenumi.\alph*}

\begin{document}	
	\begin{abstract}
		
		We prove that for every non-elementary torsion-free hyperbolic group $H$ and sufficiently large odd $n>1$, the quotient $H/H^n$ has cost $1$. In particular, when $H$ is the free group of rank $m\ge 2$, we obtain that the free Burnside group $B(m,n)$ has cost $1$ for every odd exponent $n\geq 665$. It follows that $B(m,n)$ is anti-treeable, and that its first $\ell^2$-Betti number $\smash{\beta_1^{(2)}(B(m,n))}$ vanishes. The latter recovers and extends a result of Feldkamp and Kionke, who proved that $\smash{\beta_1^{(2)}(B(m,p))=0}$ for all sufficiently large \emph{prime} exponents $p$ \cite{feldkampkionke2023}.
	\end{abstract}
	
	\maketitle
	
	\section{Introduction}
	The \emph{free Burnside group} of rank $m\ge 2$ and exponent $n\ge 2$ is  the quotient
	\[B(m,n)\coloneqq F_m/\llangle g^n : g\in F_m\rrangle,
	\]
	where $F_m$ denotes the free group of rank $m$ and $\llangle g^n : g\in F_m\rrangle$ is the normal closure in $F_m$ of all $n$-th powers. In 1902, Burnside asked whether these groups are necessarily finite \cite{Burnside1902}. This question remained open for more than sixty years. In 1968, Novikov and Adian proved that $B(m,n)$ is infinite for every $m\geq 2$ and every sufficiently large odd exponent $n$ \cite{NovikovAdjan1968}. Their original bound was later improved by Adian to odd $n\geq 665$ \cite{Adian1979}. More recently, Atkarskaya, Rips and Tent lowered the bound to $n\ge 557$, which is currently the best general bound for odd exponent \cite[Theorem 2.1]{AtkarskayaRipsTent2024}. The case of even exponent was settled for sufficiently large exponents by Ivanov and Lysenok \cite{Ivanov1994,Lysenok1996}. Nevertheless, the picture for small exponents is still far from complete. For instance, it is still unknown whether
	$B(2,5)$ is finite.
	
	Since the solution of the Burnside problem in the 60's, free Burnside groups have been extensively studied from the combinatorial and geometric points of view. Important contributions include the theory of graded diagrams and the study of the subgroup structure of $B(m,n)$ \cite{Adian1979,Olshanskii1991,Ivanov1994,Ivanov2003}, as well as geometric approaches to periodic quotients of hyperbolic groups and small cancellation theory over Burnside groups \cite{Coulon2014,CoulonGruber2019,CoulonSteenbock2022}. Free Burnside groups also provide examples of non-amenable groups without non-abelian free subgroups. More precisely, Adian proved that $B(m,n)$ is non-amenable for every $m\geq 2$ and every odd $n\geq 665$ \cite{Adyan1982}. Osin later showed that free Burnside groups of sufficiently large odd exponent satisfy the stronger property of \emph{uniform non-amenability} \cite{Osin2007}.
	
	By comparison, much less is known about free Burnside groups from the perspective of measured group theory. Among the few results in this direction, Osajda proved that if $B(m,n)$ is infinite, then $B(m,kn)$ does not have Kazhdan's property (T) for any integer $k\geq 2$ \cite[Theorem 1]{Osajda2018}. Combined with the known infinitude results for free Burnside groups, this shows in particular that free Burnside groups of sufficiently large composite exponent do not have property (T). In a different direction, Feldkamp and Kionke proved that, for all sufficiently large prime exponents $p$, the first $\ell^2$-Betti number of $B(m,p)$ vanishes \cite[Corollary 0.2]{feldkampkionke2023}.
	
	The open question most closely related to the present paper concerns the theory of \emph{orbit equivalence}. Recall that two probability-measure preserving (p.m.p.) actions $G\curvearrowright (X,\mu)$ and $H\curvearrowright (Y,\nu)$ are \emph{orbit equivalent} if there exist conull invariant subsets $X_0\subseteq X$ and $Y_0\subseteq Y$ and a measurable isomorphism $\theta:X_0\to Y_0$ such that $\theta(Gx)=H\theta(x)$ for every $x\in X_0$. Two countable groups are called orbit equivalent if they admit essentially free, ergodic, p.m.p.\ actions that are orbit equivalent. To our knowledge, the following basic question is open:
	
	\begin{question}\label{question: orbit equivalent}
		Can two non-isomorphic infinite free Burnside groups $B(m,n)$ and $B(k,\ell)$ be orbit equivalent?
	\end{question}
	
	One way to approach this question is through the theory of \emph{cost}, one of the main numerical invariants in orbit equivalence theory. Introduced by Levitt \cite{Levitt1995} and further developed by Gaboriau \cite{gaboriau2000cout}, the cost of a free p.m.p.\ action is an invariant of its orbit-equivalence relation and takes values in $[1,\infty]$ for infinite groups. For this paper, it will be convenient to use the probabilistic characterization of the cost of a group $G$ as one half of the
	infimum expected degree at the identity among all $G$-invariant
	random connected spanning graphs of $G$; see \cite[Proposition 29.5]{KechrisMiller2004} and Definition \ref{defn: cost} below. 
	
	Gaboriau's computation of cost for free groups shows that every free p.m.p.\ action of the free group $F_m$ has cost $m$ \cite[Corollaire 1]{gaboriau2000cout}. In particular, free groups have \emph{fixed price}, meaning that all their free p.m.p.\ actions have the same cost, and cost distinguishes free groups of different ranks up to orbit equivalence. This is very different from the situation for amenable groups: by the Ornstein--Weiss theorem, any two free, ergodic, p.m.p.\ actions of any two countably infinite amenable groups are orbit equivalent \cite{OrnsteinWeiss1980}. In particular, every infinite amenable group has fixed price $1$.
	\subsection{The main result}
	Our theorem concerns the more general setting of periodic quotients $H/H^n$, where $H$ is a non-elementary torsion-free hyperbolic group and $H^n\coloneqq\langle h^n\mid h\in H\rangle$ is the (normal) subgroup generated by all $n$-th powers. Ol'shanskii proved, using graded small-cancellation theory that $H/H^n$ is infinite for every sufficiently large odd integer $n$, with the threshold depending on $H$ \cite{Olshanskii1991periodic}. An alternative proof using geometric small-cancellation theory was developed by Delzant and Gromov \cite{DelzantGromov2008}; a detailed account of this approach for periodic quotients of torsion-free hyperbolic groups is given by Coulon \cite{Coulon2014}.
	
	Motivated by Question \ref{question: orbit equivalent}, it is therefore natural to ask whether the group cost can 
	distinguish free Burnside groups of different ranks or exponents. Our main result shows that, for sufficiently large odd exponent, it cannot: all free Burnside groups covered by our theorem have cost $1$, independently of their rank or exponent.

	\begin{thm}\label{thm: free burnside main theorem}
		Let $H$ be a non-elementary torsion-free hyperbolic group. Then for every sufficiently large odd $n> 1$, the quotient $H/H^n$ has cost $1$. In the case of free Burnside groups, $\mathrm{cost}(B(m,n))=1$ for every $m\geq 2$ and odd $n\geq 665$.
	\end{thm}
	
	The condition that $n$ be sufficiently large comes from Lemma \ref{lem:main_hyperbolic} below. With the currently available
	results on free Burnside groups, one may take odd $n\geq 665$. 
	
	We emphasise that Theorem~\ref{thm: free burnside main theorem} concerns the \emph{infimum} of the costs of the free p.m.p.\ actions of $B(m,n)$; it does not assert that \emph{every} such action has cost $1$. Thus, it remains open whether free Burnside groups of sufficiently large odd exponent have \emph{fixed price $1$}. A recent preprint of Slutsky gives a sufficient criterion for the latter property. Namely, if $G$ is an infinite finitely generated group with finite symmetric generating set $S$ containing $e_G$ and $\inf\left\{ |FSF^{-1}|/|F|^2 : e_G\in F\subseteq G, F \text{ finite} \right\}=0$, then $G$ has fixed price $1$ \cite{Slutsky2026}. This criterion recovers the main previously known classes of groups with fixed price $1$, including groups that have an infinite normal amenable subgroup \cite[VI.26.\ (a)]{gaboriau2000cout}, lattices in higher-rank real semisimple Lie groups \cite[Theorem D]{Fraczykmellickwilkens} and products of infinite groups \cite{Khezeli2026,Bevilacqua2025}, and also yields new examples.
	
	It is not clear whether Slutsky's criterion can be verified for free Burnside groups. In fact, results of Coulon and Steenbock on product-set growth suggest that this may be challenging. They show  in \cite[Theorem 1.2]{CoulonSteenbock2022} that there exists a constant $c>0$ such that, for free Burnside groups $B(m,n)$ of sufficiently large odd exponent, every finite subset $F\subseteq B(m,n)$ not contained in a finite cyclic subgroup satisfies $|F^r|\geq (c|F|)^{\lfloor (r+1)/2\rfloor}$ for every $r\geq 1$. Thus, finite subsets of $B(m,n)$ necessarily exhibit substantial product-set growth, which appears to be in tension with the small-product-set phenomenon underlying Slutsky's criterion. Whether these growth estimates can be used to rule out the criterion, or whether free Burnside groups nevertheless have fixed price $1$, remains open.
	
	It is also instructive to compare free Burnside groups with groups having Kazhdan's property (T). Hutchcroft and Pete proved that every infinite countable group with property (T) has cost $1$ \cite{hutchcroftpete2020}. On the other hand, infinite property (T) groups display a strong orbit-equivalence rigidity. In particular, Popa's orbit-equivalence superrigidity theorem implies that if $G$ has property (T) and no non-trivial finite normal subgroup, then its Bernoulli shift is orbit-equivalence superrigid: whenever a Bernoulli action of $G$ is orbit equivalent to a free ergodic p.m.p.\ action of a countable group $H$, the groups $G$ and $H$ are isomorphic and, after identifying them through this isomorphism, the two actions are conjugate \cite{Popa2007}.
	
	No analogous orbit-equivalence superrigidity result is currently known for free Burnside groups. Thus, even in view of Theorem \ref{thm: free burnside main theorem}, it remains open whether two non-isomorphic infinite free Burnside groups can admit orbit-equivalent free p.m.p.\ actions. To the best of our knowledge, one cannot even exclude the possibility that the Bernoulli shifts of two non-isomorphic free Burnside groups are orbit equivalent. This provides a more specific, and a priori weaker, form of Question \ref{question: orbit equivalent}. 
	By contrast, continuous orbit equivalence of Bernoulli shifts does distinguish non-isomorphic infinite free Burnside groups. Namely, if two Bernoulli shifts over $B(m,n)$ and $B(k,\ell)$ are continuously orbit equivalent and both groups are infinite, then $m=k$ and $n=\ell$. Indeed, extending earlier work of Chung and Jiang \cite{ChungJiang2017}, Cohen proved that full shifts over one-ended groups are continuously cocycle superrigid \cite{Cohen2020}: every continuous cocycle with values in a countable group is cohomologous to a homomorphism. Since every infinite finitely generated torsion group is one-ended, this applies in particular to infinite free Burnside groups. Using moreover that such groups have no non-trivial finite normal subgroups, one concludes by standard arguments that the two groups are isomorphic. Comparing their exponents and abelianizations then yields $n=\ell$ and $m=k$.
	
	\subsection{Two consequences of Theorem \ref{thm: free burnside main theorem}}
	
	The theory of $\ell^2$-Betti numbers of groups, originating in the work of Atiyah \cite{Atiyah1976} and further developed by Cheeger and Gromov \cite{CheegerGromov1986}, was extended to measured equivalence relations by Gaboriau \cite{gaboriau2002}. One of the consequences of his work is the inequality $\smash{\cost(G)\geq 1+\beta_1^{(2)}(G)}$ for every infinite countable group $G$ \cite[Corollaire 3.23]{gaboriau2002}, where $\beta_1^{(2)}(G)$ denotes the first $\ell^2$-Betti number of $G$. It remains an open problem whether this inequality can be strict for some countable group. Combining Gaboriau's inequality with Theorem \ref{thm: free burnside main theorem}, we obtain the following corollary.
	
	\begin{cor}\label{cor: betti number}
		Let $H$ be a non-elementary torsion-free hyperbolic group and let $n$ be a sufficiently large odd positive integer. Then the group $H/H^n$ has a vanishing first $\ell^2$-Betti number. This holds, in particular, for free Burnside groups of rank $m\ge 2$ and odd exponent $n$ sufficiently large.
	\end{cor}
	
	In particular, this recovers the result of Feldkamp and Kionke \cite{feldkampkionke2023}, who proved the vanishing of the first $\ell^2$-Betti number of free Burnside groups on $m\geq 2$ generators and of sufficiently large prime exponent $p$. Osajda's result \cite{Osajda2018} leaves open the possibility that free Burnside groups of prime exponent may have property (T), in which case the vanishing result of Feldkamp and Kionke would follow from it. Corollary \ref{cor: betti number} extends the vanishing of the first $\ell^2$-Betti number to all sufficiently large odd exponents, including the ones for which Osajda's result implies that the corresponding free Burnside groups do not have property (T).
	
	Another consequence of Theorem \ref{thm: free burnside main theorem} concerns the Borel graph structures that can generate the orbit equivalence relations arising from free p.m.p.\ actions of free Burnside groups. Recall that a group is called \emph{anti-treeable} if no free p.m.p.\ action of the group has an orbit equivalence relation generated by a Borel graphing whose connected components are trees; see \cite[page 44]{gaboriau2000cout}. Equivalently, a group $G$ is anti-treeable if the space of trees with vertex set $G$ admits no $G$-invariant probability measure; see \cite{PemantlePeres2000} and the comment at the end of the first paragraph of Section 0.3 in \cite{Gaboriau2005}.
	
	Gaboriau proved that every non-amenable group of cost $1$ is anti-treeable \cite[Théorème 4]{gaboriau2000cout}. The combination of this result with Theorem \ref{thm: free burnside main theorem} and the non-amenability of free Burnside groups of sufficiently large odd exponent yields the following.
	
	\begin{cor}\label{cor: antitreeable}
		Let $H$ be a non-elementary torsion-free hyperbolic group and let $n$ be a sufficiently large odd positive integer. Then the group $H/H^n$ is anti-treeable. This applies, in particular, for free Burnside groups of rank $m\ge 2$ and odd exponent $n\geq 665$.
	\end{cor}
	
	\subsection{A normal subgroup theorem}
	It remains open whether there exists a finitely generated group of bounded exponent with cost strictly greater than $1$. For groups in which every nontrivial element has finite centralizer, our results yield the following necessary condition.
	
	\begin{thm}\label{thm: normal subgroup theorem}
		Let $G$ be a countable group of bounded exponent such that the centralizer of every non-trivial element is finite. Suppose that $\mathrm{cost}(G)>1$. Then every non-trivial normal subgroup of $G$ is of finite index. Furthermore, every ergodic p.m.p.\ action of $G$ on a non-atomic probability space is essentially free.
	\end{thm}
	
	Recall that, for a prime $p$, a Tarski $p$-monster is an infinite group where every nontrivial proper subgroups is cyclic of order $p$. Ol'shanskii proved that such groups exist for every sufficiently large prime \cite{Olshanskii1982bounded}, and his construction applies in particular when $p>10^{75}$. This was later improved by Adjan and Lys\"enok to hold for all prime $p\ge 1009$ \cite{AdjanLysenok}. Every Tarski $p$-monster is simple and is generated by any two elements that do not belong to a common cyclic subgroup.
	
	The proof of Theorem \ref{thm: free burnside main theorem} does not extend directly to Tarski monsters. A crucial step is the existence of suitable infinite proper subgroups of periodic quotients of hyperbolic groups. Tarski monsters have no infinite proper subgroups.
	
	In general, for finitely generated groups that are of prime exponent, we prove in Theorem \ref{thm: max cost} an upper bound for the maximum value of any of their free p.m.p. actions. In particular, since Tarski monsters are $2$ generated, Theorem \ref{thm: max cost} improves the trivial upper bound of $2$ to an upper bound of $2-2/p$. There are also Tarski $p$-monsters with Kazhdan's property (T): for sufficiently large primes $p$, such examples can be obtained using Ol'shanskii's quotient constructions, starting from a torsion-free non-elementary hyperbolic group with property (T) \cite{Olshanskii1991periodic}. Since property (T) passes to quotients, the resulting monsters have cost $1$ \cite{hutchcroftpete2020}. However, we do not know whether any of these examples has fixed price $1$. 
	
	We finish the introduction by asking:
	\begin{question}\label{question: torsion cost larger than 1}
		Does there exist a countable group $G$ of bounded exponent such that $\mathrm{cost}(G)>1$?
	\end{question}
	
	If the bounded-exponent assumption is dropped, there are torsion groups with cost strictly greater than $1$. Indeed, L\"uck and Osin proved that, for every prime $p$, there exists a finitely generated infinite residually finite $p$-group $G$ with
	$\beta_1^{(2)}(G)>0$ \cite[Theorem~1.2]{luckosin2011}. The inequality $\smash{\mathrm{cost}(G)\geq 1+\beta_1^{(2)}(G)}$ then gives $\mathrm{cost}(G)>1$ \cite{gaboriau2002}.
	
	\subsection*{Acknowledgments}
	We are grateful to Martín Gilabert Vio for a careful reading of an earlier version of this preprint and to Raz Slutsky for comments and discussion regarding the potential applications of his recent criterion for fixed price 1. We thank Konrad Wróbel for carefully reading an earlier draft and for pointing out the comparison with Tucker-Drob’s use of centralizers. We are very grateful to an anonymous referee that suggested improvements to the statement of Proposition \ref{prop: main prop}.
	
	The authors acknowledge support of the Deutsche Forschungsgemeinschaft (DFG, German Research Foundation) under Germany's Excellence Strategy EXC 2044/2 –390685587, Mathematics Münster: Dynamics–Geometry–Structure. 

    \subsection*{AI statement}
    GPT 5.6 Sol and GPT 6 Astra were used for proofreading and to improve the presentation of the document. The mathematical arguments and the writing process were carried out entirely by the authors.
	
	\section{The proof of the main theorem}
	\subsection{Notation}
	
	We begin by fixing some notation and conventions concerning graphs. All graphs considered in this paper are simple and undirected. If $V$ is a countable set, we denote by
	\[
	[V]^2\coloneqq\bigl\{\{x,y\}:x,y\in V,\ x\neq y\bigr\}
	\]
	the set of unordered pairs of distinct elements of $V$. A graph $\mathcal{G}$ with vertex set $V$ is identified with its edge set $E(\mathcal{G})\subseteq [V]^2$. In particular, an edge joining $x$ and $y$ will always be denoted by
	$\{x,y\}$; thus $\{x,y\}=\{y,x\}$.
	
	A finite sequence of vertices $x_0,x_1,\ldots,x_k$ is called a path in $\mathcal{G}$ from $x_0$ to $x_k$ if
	\[
	\{x_{i-1},x_i\}\in E(\mathcal{G})
	\ \ \ \text{for every }1\leq i\leq k.
	\]
	Two vertices $x,y\in V$ are said to be connected in $\mathcal{G}$ if
	there exists a path in $\mathcal{G}$ from $x$ to $y$. The graph $\mathcal{G}$ is connected if every two vertices of $V$ are connected.
	
	For $A\subseteq V$, we denote by $\mathcal{G}[A]$ the subgraph of $\mathcal{G}$ induced by $A$, so that $E(\mathcal{G}[A]) = E(\mathcal{G})\cap [A]^2$. Thus, when we say that $\mathcal{G}$ is connected on $A$, we mean that the induced subgraph $\mathcal{G}[A]$ is connected.
	
	For $x\in V$, the degree of $x$ in $\mathcal{G}$ is
	\[
	\deg_{\mathcal G}(x)\coloneqq \left|\{y\in V:\{x,y\}\in E(\mathcal G)\}\right|,
	\]
	where the degree is allowed to be infinite.
	
	We denote by $\mathrm{Graph}(V)$ the space of all simple undirected graphs with
	vertex set $V$. Identifying a graph with the indicator function of its
	edge set, we identify
	\[
	\mathrm{Graph}(V)=\{0,1\}^{[V]^2},
	\]
	and endow this space with the corresponding product Borel $\sigma$-algebra. We denote by $\mathcal{S}(V)\subseteq\mathrm{Graph}(V)$ the Borel subspace consisting of connected graphs.
	
	When $V=G$ is a countable group, $G$ acts on $\mathrm{Graph}(G)$ by left
	translation: for $g\in G$ and $\mathcal G\in\mathrm{Graph}(G)$, the graph
	$g\mathcal G$ is defined by the set of edges
	\[
	E(g\mathcal G) \coloneqq \left\{\{gx,gy\}:\{x,y\}\in E(\mathcal{G})\right\}.
	\]
	The space $\mathcal{S}(G)$ is invariant under this action. We denote by $\mathcal {M}(G,\mathcal{S}(G))$ the set of $G$-invariant Borel probability measures on $\mathcal{S}(G)$.

	We will use the following characterization as our definition; see \cite[Proposition 29.5]{KechrisMiller2004} for the proof that it coincides with Gaboriau's definition of cost given in \cite[Définitions I.5 (3)]{gaboriau2000cout}.
	
	\begin{defn}\label{defn: cost}
		The \emph{cost} of a countable group $G$ is
		\begin{equation}\label{eq: cost def}
			\cost(G) = \inf\left\{ \frac{1}{2} \int_{\mathcal S(G)} \deg_{\mathcal G}(e_G)\,d\mu(\mathcal G) :\mu\in \mathcal{M}(G,\mathcal S(G))\right\}.
		\end{equation}
	\end{defn}
	
	For the foundational results on cost and their applications to orbit equivalence, we refer to Gaboriau's seminal paper \cite{gaboriau2000cout}, as well as to Part III of the book of Kechris and Miller \cite{KechrisMiller2004}, which gives a detailed account of the theory. For surveys of subsequent developments and of the role of cost in measured and asymptotic group theory, we refer to Gaboriau's ICM survey \cite{Gaboriau2010}, to Gaboriau's lecture notes \cite{Gaboriaulecturenotes} and to Abért's ICM survey \cite{AbertICM}.
	
	\subsection{Preliminary results}
	
	We will use the following theorem of Gaboriau.
	\begin{thm}[{\cite[Theorem 3.4]{gaboriau2002orbit}}]\label{thm:gaboriau}
		Let $G$ be a countable group and let $N\lhd G$ be an infinite normal subgroup such that $N$ has infinite index in $G$ and $\mathrm{cost}(N)<\infty$. Then $\mathrm{cost}(G)=1$.
	\end{thm}
	In \cite{gaboriau2002orbit} it is assumed that $N$ is finitely generated, but this can be replaced by the assumption that $N$ has finite cost; see the footnote on page 884 of \cite{hutchcroftpete2020} and \cite[Theorem 2.56]{Gaboriaulecturenotes}
	
	Thanks to Theorem \ref{thm:gaboriau}, in order to establish Theorem \ref{thm: free burnside main theorem} it suffices to find an infinite normal subgroup of the free Burnside group $B(m,n)$ that has infinite index and finite cost. This is guaranteed by Lemma \ref{lem:main_hyperbolic} below. Before stating it, we state some useful preliminary results that will be used later.
	We denote by $C_G(g)$ the centralizer of $g\in G$.
	
	The following lemma is a compilation of known results regarding free Burnside groups and periodic quotients of hyperbolic groups.
	
	\begin{lem}\label{lem:Burnside and hyperbolic factos}
		Let $H$ be a torsion-free non-elementary hyperbolic group. There exists $n(H)>1$ such that, for every odd $n\geq n(H)$, 
		\begin{enumerate}
			\item the quotient $H/H^n$ is infinite, and
			\item for every non-trivial $h\in H/H^n$, the centralizer $C_{H/H^n}(h)$ is finite.
		\end{enumerate}
		In addition, for every non-trivial $h\in H$, there exists $n_h(H)>1$ such that for every odd $n\geq n_h(H)$, the element $h$ maps to a non-trivial element under the quotient map $H\to H/H^n$. In the case $H=F_m$ for $m\geq 2$ (i.e., $H/H^n$ is the free Burnside group $B(m,n)$), then one can take $n(H)=665$ above.
	\end{lem}
	\begin{proof}
		The first item follows from the main theorem in \cite{Olshanskii1991periodic}. The second item follows from the arguments in the last paragraph of \cite[p. 540]{Olshanskii1991periodic}. The assertion that every non-trivial $h\in H$ maps to a non-trivial element under the quotient map $H\to H/H^n$ for sufficiently large odd $n>1$ is argued in the last paragraph of the proof of the main theorem in \cite[\S 10]{Olshanskii1991periodic}.
		
		For the case $H=F_m$, $m\geq 2$ and $n\geq 665$, the first item was proven by \cite{Adian1979}, while the second item can be found, for example, in \cite[Theorem VI.3.2]{Adian1979}, \cite[Theorem 19.5]{Olshanskii1991}, or \cite[Lemma 4.18]{OlshanskiiSapir2002}.
	\end{proof}
	
	A component of the proof of Theorem \ref{thm: free burnside main theorem} is that centralizers of non-trivial elements are finite; see Lemma \ref{lem:Burnside and hyperbolic factos} above. Centralizers also play a role in Tucker-Drob's proof that inner amenable groups have fixed price $1$ \cite[Theorem 5]{tuckerdrob2020}, although in a different manner. His argument uses the fact that for every non-amenable subgroup $L\leq G$, the intersection $L\cap C_G(g)$ is non-amenable for almost every $g\in G$ with respect to an atomless conjugation-invariant mean \cite[Lemma 4.2]{tuckerdrob2020}. Thus, Tucker-Drob's proof relies on large, non-amenable centralizers, whereas ours relies on centralizers being finite of a uniformly bounded size.
	
	\begin{lem}\label{lem:main_hyperbolic}
		Let $H$ be a non-elementary torsion-free hyperbolic group without property (T). Then, for sufficiently large odd $n\geq 1$ the group $H/H^n$ contains an infinite normal subgroup of infinite index. Moreover, if $H=F_m$ for some $m\geq 3$, it suffices to take $n\geq 665$.
	\end{lem}
	
	\begin{proof}
		We start by claiming that there exists a non-elementary torsion-free hyperbolic quotient $Q$ of $H$ such that the quotient map $\pi\colon H\to Q$ has non-trivial kernel. Indeed, take a non-elementary torsion-free hyperbolic group $H'$ with property (T). It follows from \cite[Theorem 2]{Olshanskii1993}, there exists a non-elementary torsion-free hyperbolic group $Q$ that is a quotient of both $H$ and $H'$. Let $\pi\colon H\to Q$ be the quotient homomorphism. Since $Q$ is a quotient of $H'$, it has property (T) and therefore $\ker(\pi)\neq \{e_H\}$. Fix a non-trivial element $h\in \ker(\pi)$.
		
		Consider the integers $n(H)$, $n_h(H)$ and $n(Q)$ granted by Lemma \ref{lem:Burnside and hyperbolic factos}, and an arbitrary odd number $n\geq \max\{n(H),n_h(H),n(Q)\}$.  Set $G:=H/H^n$, and let $g$ be the image of $h$ under the quotient map $H\to G$, which is non-trivial by our choice of $n$. We will see that the existence of $\pi$ implies the existence of an infinite normal subgroup $N\lhd G$ of infinite index. Since 
		\[H^n\leqslant \ker(H\overset{\pi}{\to}Q\to Q/Q^n),\]
		we obtain that $\pi$ induces a quotient map $\pi_n\colon G\to Q/Q^n$. Because $h\in \ker(\pi)$, we get that $g\in \ker(\pi_n)\backslash{e_{G}}$, so that $N:=\ker(\pi_n)$ is non-trivial. By Lemma \ref{lem:Burnside and hyperbolic factos} both groups $H/H^n$ and $Q/Q^n$ are infinite. Thus $N$ has infinite index in $G$. Assume, looking for a contradiction, that $N$ is finite. Then $G/C_{G}(N)\hookrightarrow \mathrm{Aut}(N)$ is finite as well, implying that $C_{G}(N)$ is infinite. This cannot be the case, as $C_{G}(N)\subseteq C_{G}(g)$, and by Lemma \ref{lem:Burnside and hyperbolic factos} the subgroup $C_{G}(g)$ is finite. 
		
		Assume now that $H=F_m$ with $m\geq 3$, and let $x_1,\dots,x_m$ be a freely generating set for $F_m$. Consider the quotient map $\pi\colon F_m\to F_{m-1}$ that sends $x_m$ to the identity. As $F_{m-1}$ is a non-elementary torsion-free hyperbolic group, and since $x_m\in\ker(\pi)$ maps to a non-trivial element of $B(m,n)$, the above argument shows that there is an infinite normal subgroup $N$ of infinite index in $B(m,n)=F_m/F_m^n$ for $n\geq \max\{n(F_m),n(F_{m-1})\}$. By Lemma \ref{lem:Burnside and hyperbolic factos}, one can take $n(F_m)=n(F_{m-1})=665$.
	\end{proof}
	
	\subsection{An invariant random connected graph with finite cost}
	We next show that the presence of an infinite finitely generated subgroup, together with the assumption of abelian centralizers of bounded order, is enough to guarantee finite cost. The idea is to start with a connected Cayley graph on an infinite finitely generated subgroup $N_0\leq N$, and then enlarge $N_0$ through an increasing sequence
	\[
	N_0\leq N_1\leq N_2\leq \cdots, \text{ with } N=\bigcup_{j\geq 0}N_j.
	\]
	At each step, we add a sparse family of random edges in order to connect the different cosets of $N_{j-1}$ inside $N_j$. The assumption on centralizers provides infinitely many independent ways of connecting any two such cosets, which allows us to use the Borel-Cantelli lemma to obtain connectivity almost surely. Choosing the probabilities of the edges to decay sufficiently fast ensures at the same time that the resulting invariant random spanning graph has finite expected degree.
	\begin{prop}\label{prop: main prop}
		Let $G$ be a countable group such for some $n\ge 1$, the order of every element is at most $n$. Suppose that the centralizer $C_G(g)$ of
		every non-trivial element $g\in G$ is finite. Then
		\begin{enumerate}
			\item\label{item:item1} every subgroup of $G$ has finite cost, and
			\item if $G$ is infinite, then $\mathrm{cost}(G)=\min\left\{\mathrm{cost}(N): \text{infinite } N\leqslant G \right\}$.
		\end{enumerate}
		
	\end{prop}
	\begin{proof}
		Let $N\leqslant G$ be any subgroup. If every infinite subgroup of $N$ has infinite cost, then $N$ is locally finite, so it is amenable and consequently has finite cost. Hence, to prove item \eqref{item:item1} in the statement of the proposition, we may assume that $N$ contains an infinite subgroup $N_0\leqslant N$ of finite cost.
		
		Let $0<\varepsilon<1$. In what follows we will define a probability measure $\mu\in \mathcal{M}(N,\mathcal{S}(N))$ such that 
		\begin{equation}\label{eq: desired cost upper bound}
			\frac{1}{2}\int_{\mathcal{S}(N)}\deg_{\mathcal{G}}(e_N)\ d\mu(\mathcal{G})\leq \mathrm{cost}(N_0)+\varepsilon.
		\end{equation}
		
		Let $\nu\in \mathcal{M}(N_0,\mathcal{S}(N_0))$ be a probability measure of average degree
		\[\frac{1}{2}\int_{\mathcal{S}(N_0)}\mathrm{deg}_{\mathcal{G}}(e_{N_0})\,d\nu(\mathcal{G})\leq \mathrm{cost}(N_0)+\frac{\varepsilon}{2}.\]
		Let us choose a countable subset $J\subseteq \N$ and a sequence $(a_j)_{j\in J}\subseteq N$ such that, defining inductively $N_j:=\langle N_{j-1},a_j\rangle$, we have $a_j\notin N_{j-1}$ for every $j\in J$ and $N=\bigcup_{j\in J}N_j$.
		
		Consider the Borel probability measure $\widetilde{\P}\in\mathrm{Prob}(\{0,1\}^{J\times N})$ given by
		\[
		\widetilde{\P}\coloneqq  \bigotimes_{j\in J} \bigotimes_{g\in N} \left( \left(1-\frac{\varepsilon}{2^{j+1}}\right)\delta_0+\frac{\varepsilon}{2^{j+1}}\delta_1
		\right).
		\]
		Thus, the random variables
		\begin{align*}
			\{0,1\}^{\N_\geq 1\times N}&\to \{0,1\}\\
			\omega&\mapsto \omega(j,g),
		\end{align*}
		for $j\in J$ and $g\in N$, are independent and satisfy
		\[\widetilde{\P}\left(\omega(j,g)=1\right)=\frac{\varepsilon}{2^{j+1}}=1-\widetilde{\P}\left(\omega(j,g)=0\right).\]
		
		Choose now a set $R$ of representatives of left cosets of $N_0$ in $N$. Let 
		\[\Omega:=\mathrm{Graph}(N_0)^R\times \{0,1\}^{J\times N},\]
		and define a map $\Psi:\Omega\to \mathrm{Graph}(N)$ as follows. For every $\omega:=((\mathcal{G}_r)_{r\in R},x)\in\Omega$, define a graph $\Psi(\omega)$ with vertex set $N$
		by declaring its set of edges to be
		\[
		E(\Psi(\omega))	\coloneqq \bigsqcup_{r\in R}E(r\mathcal{G}_r)\cup \bigcup_{j\in J} \left\{
		\{g,ga_j\}:g\in N,\ x(j,g)=1	\right\}.\]
		
		The map $\Psi:\Omega\to\mathrm{Graph}(N)$ is measurable. Indeed, recall that $\mathrm{Graph}(N)=\{0,1\}^{[N]^2}$ is endowed with the product Borel $\sigma$-algebra. It therefore suffices to verify measurability in each coordinate $e\in[N]^2$. For such an edge $e$, if $e=\{sh,sh'\}$ for some $s\in R$ and $h,h'\in N_0$, we have 
		\[\{\omega\in\Omega:e\in E(\Psi(\omega))\}=\{(\mathcal{G}_r)_{r\in R}\in \mathrm{Graph}(N_0):\{h,h'\}\in E(\mathcal{G}_{s}))
		\}\times\{0,1\}^{J\times N},\]
        which is measurable. On the other hand, if $e=\{g,ga_j\}$ for some $g\in N$ and $j\in J$, then
		\[\{\omega\in\Omega:e\in E(\Psi(\omega))\}=\mathrm{Graph}(N_0)^R\times \bigcup_{\substack{j\in J,\ g\in N\\ \{g,ga_j\}=e}} \{x\in\{0,1\}^{J\times N}:x(j,g)=1\}.\]
		The latter is a countable union of cylinder events, and is therefore
		measurable. Thus $\Psi$ is measurable.
		
		We set $\P\coloneqq \nu^{\otimes R}\otimes \widetilde{\P}$ and $\mu\coloneqq \Psi_*\P$. Let us first estimate the average degree of a graph sampled by $\mu$. We do this in the next claim.
		
		\begin{claim}\label{claim: average degree}
			We have
			\[
			\frac{1}{2}\int_{\mathrm{Graph}(N)} \deg_{\mathcal G}(e_N)\ d\mu(\mathcal G)\leq \mathrm{cost}(N_0)+\varepsilon,
			\]
			so that Equation \eqref{eq: desired cost upper bound} holds.
		\end{claim}
		
		\begin{proof}
			For each $j\in J$, an edge arising at level $j$ can be incident to $e_N$ only through one of the two coordinates $(j,e_N)$ or $(j,a_j^{-1})$. Therefore, we get
			\begin{align*}
				\int_{\mathrm{Graph}(N)} \deg_{\mathcal{G}}(e_N)\ d\mu(\mathcal{G}) &= \mathbb{E}_{\P}\left[ \deg_{\Psi(\omega)}(e_N)\right]\\
				&=\sum_{g\in N}\P\left(\{\omega\in \Omega:\{e_N,g\}\in \Psi(\omega)\}\right)\\
				&\leq\sum_{g\in N_0}\nu(\{\mathcal{G}\in\mathrm{Graph}(N_0):\{e_{N_0},g\}\in E(\mathcal{G})\})\\
				&\qquad +\sum_{j\in J}\Big{(}\widetilde{\mathbb{P}}(x(j,e_N)=1) + \widetilde{\mathbb{P}}(x(j,a_j^{-1})=1)\Big{)}\\
				&= \int_{\mathrm{Graph}(N_0)}\mathrm{deg}_{\mathcal{G}}(e_{N_0})d\nu(\mathcal{G})+2\sum_{j\in J}\frac{\varepsilon}{2^{j+1}}\\
				&\leq 2\left(\mathrm{cost}(N_0)+\frac{\varepsilon}{2}\right)+\varepsilon
				=2(\mathrm{cost}(N_0)+\varepsilon).
			\end{align*}
			This proves the claim.
		\end{proof}
		
		\begin{claim}
			The probability measure $\mu\in \mathrm{Prob}(\mathrm{Graph}(N))$ is $N$-invariant.
		\end{claim}
		\begin{proof}
            Let $\rho\colon N\to R$ and $\theta\colon N\to N_0$ be the two maps satisfying $h=\rho(h)\theta(h)$ for every $h\in N$, which are uniquely defined from the fact that $N=\bigsqcup_{r\in R}rN_0$. Consider the action $N\curvearrowright \{0,1\}^{J\times N}$ given by $(hx)(j,g)=x(j,h^{-1}g)$ for all $h,g\in N$ and $j\in J$, and the action $N\curvearrowright \mathrm{Graph}(N_0)^R$ given by $h\cdot (\mathcal{G}_r)_{r\in R}=(\mathcal{G}_r')_{r\in R}$, where \[\mathcal{G}_{\rho(hr)}'=\theta(hr)\mathcal{G}_r,\quad r\in R.\]
            Since $r\mapsto \rho(hr)$ defines a permutation of $R$, the element $(\mathcal{G}_r')_{r\in R}$ above is well-defined. Let us show that the definition really gives an action. Let $(\mathcal{G}_r)_{r\in R}\in \mathrm{Graph}(N_0)^R$ and $h_1,h_2\in N$. Write $(\mathcal{G}_r')_{r\in R}=h_1\cdot (\mathcal{G}_r)_{r\in R}$, $(\mathcal{G}_r'')_{r\in R}=h_2\cdot (\mathcal{G}_r')_{r\in R}$ and $(\mathcal{G}_r''')_{r\in R}=h_2h_1\cdot (\mathcal{G}_r)_{r\in R}$. Then for every $r\in R$ one has
            \[\mathcal{G}''_{\rho(h_2\rho(h_1r))}=\theta(h_2\rho(h_1r))\mathcal{G}_{\rho(h_1r)}'=\theta(h_2\rho(h_1r))\theta(h_1r)\mathcal{G}_r.\]
            On the other hand, we have $\mathcal{G}_{\rho(h_2h_1r)}'''=\theta(h_2h_1r)\mathcal{G}_r$. Since $h_2h_1r=\rho(h_2h_1r)\theta(h_2h_1r)$ and
            \[\rho(h_2\rho(h_1r))\theta(h_2\rho(h_1r))\theta(h_1r)=h_2\rho(h_1r)\theta(h_1r)=h_2h_1r,\]
            the uniqueness of the decomposition of $h_2h_1r$ into a product in $RN_0$ gives $\rho(h_2h_1r)=\rho(h_2\rho(h_1r))$ and $\theta(h_2h_1r)=\theta(h_2\rho(h_1r))\theta(h_1r)$. Therefore $(\mathcal{G}''_r)_{r\in R}=(\mathcal{G}'''_r)_{r\in R}$, proving that the definition gives an action of $N$ on $\mathrm{Graph}(N_0)^R$.

            Considering the product action $N\curvearrowright \Omega$, we get, by construction, that $\Psi\colon \Omega\to \mathrm{Graph}(N)$ is $N$-equivariant. Since $\widetilde{\P}\in \mathrm{Prob}(\{0,1\}^{J\times N})$ is $N$-invariant, so is $\P\in \mathrm{Prob}(\Omega)$. This yields the desired $N$-invariance of $\mu=\Psi_*\P$.
		\end{proof}

		We will now show that for $\mathbb{P}$-almost every $\omega\in \Omega$, the graph $\Psi(\omega)$ is connected. In order to do this, it suffices to prove the following claim.
		
		\begin{claim}
			For every $j\in J$, for $\P$-almost every $\omega\in\Omega$ and every $g\in N$, the induced subgraph $\Psi(\omega)[gN_j]$ on the coset $gN_j$ is connected.
		\end{claim}
		
		\begin{proof}[Proof of claim]
			
			For $j=0$, the conclusion holds since
			for every $g\in N$, the induced subgraph $\Psi(\omega)[gN_0]$ is sampled according to the distribution $\nu\in \mathcal{M}(N_0,\mathcal{S}(N_0))$, which is supported on connected subgraphs of $N_0$.
			
			Let $j\geq 1$ and assume that the claim holds for $j-1$. Since $N_j=\langle N_{j-1},a_j\rangle$, and since by the induction hypothesis every left coset of $N_{j-1}$ is connected, it is enough to show that for $\mathbb P$-almost every $\omega\in\Omega$ and every $g\in N$, the vertices $g$ and $ga_j$ are connected in $\Psi(\omega)[gN_j]$.
			
			Set $c_i\coloneqq a_jb_i$. Notice first that $a_jb\neq e_G$ for every $b\in N_{j-1}$, since $a_j\notin N_{j-1}$. We claim that there exists an infinite sequence $(b_i)_{i\geq 0}$ in $N_{j-1}$ such that the sets
			\[\langle c_i\rangle\backslash\{e_G\}, \text{ for } i\geq 0,\]
			are pairwise disjoint. Let us suppose, looking for a contradiction, that no such infinite sequence exists. We may then choose a maximal finite family $b_0,\ldots,b_k\in N_{j-1}$ such that
			\[ \langle c_i\rangle\backslash\{e_G\},
			\text{ for } 0\leq i\leq k,
			\]
			are pairwise disjoint. By maximality, for every $b\in N_{j-1}$ there
			exists $0\leq i\leq k$ such that
			\[
			\langle a_jb\rangle\cap \langle c_i\rangle\neq\{e_G\}.
			\]
			In particular, there exist $k_i,\ell_{i}\geq 1$ such that $e_G\neq (a_jb)^{k_i}=c_i^{\ell_{i}}$, yielding $a_jb\in C_G((a_jb)^{k_i})=C_G(c_i^{\ell_{i}})$. This implies that $b\in a_j^{-1}C_G(c_i^{\ell_{b,i}})$, and it follows that
			\[N_{j-1}\subseteq\bigcup_{i=0}^k\bigcup_{\ell=1}^{d_i-1} a_j^{-1}C_G(c_i^{\ell_i}),\]
            where $d_i$ is the order of the element $c_i\in N$.	The right-hand side is finite by our assumption of finiteness of centralizers of non-trivial elements, whereas $N_{j-1}$ is infinite. This is a contradiction, and hence we conclude the existence of an infinite sequence $(b_i)_{i\geq 0}$ in $N_{j-1}$ such that the sets $(\langle a_jb_i\rangle\backslash\{e_G\})_{i\ge 0}$ are pairwise disjoint.
			
			Fix $g\in N$. For each $i\geq 0$, define
			\[
			A_{i,g}\coloneqq \left\{\omega=((\mathcal{G})_{r\in R},x)\in\Omega: x(j,gc_i^\ell)=1 \text{ for every }1\leq \ell\leq d_i-1 \right\}.
			\]
			The events $A_{i,g}$, $i\geq 0$, depend on pairwise disjoint collections
			of coordinates of $\Omega$, and are therefore independent. Moreover,
			\[\mathbb P(A_{i,g})=\left(\frac{\varepsilon}{2^{j+1}}\right)^{d_i-1}\geq\left(\frac{\varepsilon}{2^{j+1}}\right)^{n-1}.\]
			Consequently, $\sum_{i\geq 0}\mathbb P(A_{i,g})=\infty$. By the Borel-Cantelli lemma, for $\mathbb{P}$-almost every $\omega\in \Omega$, the events $(A_{i,g})_{i\ge 0}$ occur infinitely often. In particular, for $\Omega_{j,g}\coloneqq \bigcup_{i\geq 0}A_{i,g}$ we have $\mathbb{P}(\Omega_{j,g})=1$, and since $N$ is countable, we obtain that $\Omega_j\coloneqq\bigcap_{g\in N}\Omega_{j,g}$ has full $\mathbb{P}$-measure.
			
			Consider the subset of $\Omega_{j-1}\subseteq\Omega$ of full $\mathbb{P}$-measure where $\Psi(\omega)[gN_{j-1}
			]$ is connected for every $g\in N$. Fix $\omega\in \Omega_{j}\cap \Omega_{j-1}$ and $g\in N$. Then $\omega \in A_{i,g}$ for some value $i\ge 0$. We claim that $g$ and $ga_j$ are connected in $\Psi(\omega)[gN_j]$. Indeed, since $b_i\in N_{j-1}$, the vertices $ga_j$ and $gc_i=ga_jb_i$ belong to the same left coset of $N_{j-1}$, and hence are connected inside this coset by the induction hypothesis.
			
			For every $1\leq \ell\leq d_i-1$, the event $A_{i,g}$ provides the edge $\{gc_i^\ell,gc_i^\ell a_j\}$ in the graph $\Psi(\omega)[gN_j]$. Furthermore, whenever $1\leq \ell\leq d_i-2$, we have $gc_i^{\ell+1}=gc_i^\ell a_jb_i$, so the vertices $gc_i^\ell a_j$ and $gc_i^{\ell+1}$ lie in the same left coset of $N_{j-1}$ and are therefore connected by the induction hypothesis. Finally, since $c_i^{d_i}=e_G$, we have $c_i^{d_i-1}a_jb_i=e_G$,
			and hence $c_i^{d_i-1}a_j=b_i^{-1}$. Therefore, $gc_i^{d_i-1}a_j=gb_i^{-1}$ belongs to the same left coset $gN_{j-1}$ as $g$.
			
			Putting these paths together gives a path of the form
			\[
			ga_j \leadsto gc_i \rightarrow gc_i a_j \leadsto gc_i^2 \rightarrow gc_i^2a_j \leadsto \cdots \leadsto gc_i^{d_i-1} \rightarrow gc_i^{d_i-1}a_j=gb_i^{-1}\leadsto g,
			\]
			where each symbol $\leadsto$ denotes a path contained in a left coset
			of $N_{j-1}$, whose existence is guaranteed by the induction hypothesis, and each symbol $\rightarrow$ denotes one of the random edges supplied by $A_{i,g}$. All the vertices appearing in this path belong to $gN_j$.
			
			We have therefore shown that, for every $\omega\in \Omega_j\cap \Omega_{j-1}$ and every $g\in N$, the vertices
			$g$ and $ga_j$ are connected in $\Psi(\omega)[gN_j]$. Since $N_j=\langle N_{j-1},a_j\rangle$, and since every left coset of $N_{j-1}$ is connected by the induction hypothesis, it follows that $\Psi(\omega)[gN_j]$ is connected for every $g\in N$ and $\omega\in \Omega_j\cap \Omega_{j-1}$. The fact that $\mathbb{P}(\Omega_j\cap \Omega_{j-1})=1$ completes the induction and proves the claim.
		\end{proof}
		Since there are only countably many $j\geq0$, after intersecting the
		corresponding events of probability one we may assume that, simultaneously for every $j\geq0$ and every $g\in N$, the induced subgraph $\Psi(\omega)[gN_j]$ is connected. Since $N=\bigcup_{j\geq0}N_j$, it follows that every $g\in N$ belongs to some $N_j$ and is therefore connected to $e_N$ in $\Psi(\omega)$. Hence $\Psi(\omega)$ is connected for $\mathbb P$-almost every $\omega\in\Omega$. Thus, $\mu=\Psi_*\mathbb P\in\mathcal M(N,\mathcal S(N))$. Together with Claim \ref{claim: average degree}, this gives
		\[
		\mathrm{cost}(N)\leq \frac{1}{2}\int_{\mathcal{S}(N)} \deg_{\mathcal{G}}(e_N)\ d\mu(\mathcal{G}) \leq \mathrm{cost}(N_0)+\varepsilon.
		\]
        Since $\varepsilon>0$ is arbitrary, we conclude $\mathrm{cost}(N)\leq \mathrm{cost}(N_0)$, which proves item (i). Finally, if $G$ is infinite and $N_0\leqslant G$ is infinite, by item (i) $N_0$ has finite cost and from the above arguments the inequality $\mathrm{cost}(G)\leq \mathrm{cost}(N_0)$ follows, proving item (ii).
	\end{proof}
	\section{Proof of the main theorems}
	\subsection{Cost 1 for periodic quotients of hyperbolic groups}
	\begin{proof}[Proof of Theorem \ref{thm: free burnside main theorem}]
		Let $H$ be a non-elementary torsion-free hyperbolic group. If $H$ has property (T) then the group $H/H^n$ has property (T) and is infinite for sufficiently large odd $n>1$ by Lemma \ref{lem:Burnside and hyperbolic factos}, so it has cost 1 by \cite{hutchcroftpete2020}.
        
        Assume now that $H$ does not have property (T), and let $n$ be an odd integer sufficiently large so that Lemmas \ref{lem:Burnside and hyperbolic factos} and \ref{lem:main_hyperbolic} apply to $H$. Set $G\coloneqq H/H^n$.  By Lemma \ref{lem:main_hyperbolic}, there exists an infinite normal subgroup $N\lhd G$ of infinite index.
		Moreover, by Lemma \ref{lem:Burnside and hyperbolic factos}, the centralizer in $G$ of every non-trivial element is cyclic of order $n$. Hence
		Proposition \ref{prop: main prop} applies and gives $\cost(N)<\infty$. Since $N$ is infinite, normal, and of infinite index in $G$, Gaboriau's Theorem \ref{thm:gaboriau} now yields $\cost(G)=1$. This proves the theorem in the general case. 
        
        Recall that by Lemma \ref{lem:Burnside and hyperbolic factos} and Lemma \ref{lem:main_hyperbolic} it suffices to take odd $n\ge 665$ in the case $H=F_m$, $m\ge 3$. Finally, if $m=2$, Adian proved in \cite[Theorem VI.3.7]{Adian1979} that $B(2,n)$ contains a subgroup isomorphic to $B(3,n)$ for odd $n\ge 665$. Thus, for odd $n\geq 665$ we obtain from part (ii) of Proposition \ref{prop: main prop} that
        \[1\leq \mathrm{cost}(B(2,n))\leq \mathrm{cost}(B(3,n))=1.\]
	\end{proof}
	
	\begin{cor}
	The free Burnside group $B(\infty,n)$ of infinite rank and odd exponent $n\ge 665$ has cost $1$.
	\end{cor}
	\begin{proof}
		The free Burnside group $B(\infty,n)$ contains a copy of $B(2,n)$ for each $n\ge 665$. Proposition \ref{prop: main prop} ensures that $\mathrm{cost}(B(\infty,n))\le \mathrm{cost}(B(2,n))$, which is $1$ thanks to Theorem \ref{thm: free burnside main theorem}.
	\end{proof}

    \begin{proof}[Proof of Corollary \ref{cor: antitreeable}]
    Since $\mathrm{cost}(H/H^n)=1$ for all sufficiently large odd $n$ by Theorem \ref{thm: free burnside main theorem}, it follows from \cite[Théorème 4]{gaboriau2000cout} that if $H/H^n$ is non-amenable, then it is anti-treeable. The construction provided in the proof of Lemma \ref{lem:main_hyperbolic} shows that $H$ surjects onto a non-elementary torsion free hyperbolic group $Q$ with property (T). By Lemma \ref{lem:Burnside and hyperbolic factos} the quotient $Q/Q^n$ is infinite and has property (T) for sufficiently large odd $n$, implying non-amenability of $H/H^n$. In the case of free Burnside groups, Adian proved in \cite{Adyan1982} that $B(m,n)$ is non-amenable for all odd $n\geq 665$.
    \end{proof}
	
	\subsection{The normal subgroup theorem}
	
	\begin{proof}[Proof of Theorem \ref{thm: normal subgroup theorem}]
		If $G$ has bounded exponent and the centralizer of every non-trivial element is finite, Proposition \ref{prop: main prop} ensures that the cost of any subgroup of $G$ is finite.
		
		We claim that $G$ has no non-trivial finite normal subgroups. Indeed, suppose that $N$ is a finite normal subgroup of $G$. The action of $G$ on $N$ by conjugation gives a homomorphism $G\to \mathrm{Aut}(N)$ with kernel $C_G(N)\coloneqq \{g\in G:gn=ng \text{ for each }n\in N\}$. Since $N$ is finite, the subgroup $C_G(N)$ has finite index on $G$. Suppose that $N$ has a non-trivial element $n^{*}$. Then $C_G(n^{*})$ is finite, and henceforth so is $C_G(N)$. Since the latter is a finite index subgroup of $G$, this would imply that $G$ is also finite. This is a contradiction, since $\mathrm{cost}(G)>1$ whereas any finite group has cost strictly smaller than $1$. 
		
		Hence, we conclude that any non-trivial normal subgroup of $G$ is infinite. Since $\mathrm{cost}(G)>1$, any such subgroup must have finite index: otherwise, Theorem \ref{thm:gaboriau} would imply that $\mathrm{cost}(G)=1$. This proves the first statement of our theorem.

		Next, let us show the statement for ergodic p.m.p.\ actions on non-atomic probability spaces. It is proved in \cite[Theorem 1.9]{Tuckerdrob2012} that if $G$ is a group with $\mathrm{cost}(G)>1$, then for every ergodic p.m.p.\ action of $G$ on a non-atomic probability space $(X,\mu)$ exactly one of the following holds:
		\begin{itemize}
			\item for $\mu$-almost every $x\in X$, the stabilizer $G_x$ is finite.
			\item for $\mu$-almost every $x\in X$, it holds that $\mathrm{cost}(G_x)=\infty$.
		\end{itemize}
		This result follows from a version of Gaboriau's Theorem \ref{thm:gaboriau} for invariant random subgroups. The analogous generalization to IRSs in terms of the first $\ell^2$-Betti number (instead of the cost) of the group is proved in \cite[Théorème 5.4]{BergeronGaboriau2004}; see also \cite[Corollary 1.6]{SauerThom2010}.
		
		Since every subgroup of $G$ has finite cost, it must hold that almost every stabilizer is finite. Let us denote by $\mathrm{Sub}(G)$ the Chabauty space of subgroups of $G$, and consider the stabilizer map 
		\begin{align*}
			\mathcal{S}: X&\to \mathrm{Sub}(G)\\
			x&\mapsto G_x
		\end{align*}
		associated with the ergodic p.m.p.\ action $G\curvearrowright (X,\mu)$. The push-forward $\eta\coloneqq \mathcal{S}_{*}\mu$ is thus an ergodic invariant random subgroup of $G$ that is supported on finite subgroups. Since a countable group may only have countably many distinct finite subgroups, the ergodicity and $G$-invariance of $\eta$ implies that $\eta$ may be written as the uniform measure on a finite orbit. For $\{e_G\}\neq K\leqslant G$ in the support of $\eta$, we have that $N_G(K)/C_G(K)\hookrightarrow \mathrm{Aut}(K)$, thus $C_G(K)$ has finite index in $N_G(K)$, and $N_G(K)$ is the stabilizer of $K$ under conjugation by $G$, so it has finite index in $G$. It follows that $C_G(K)$ has finite index in $G$. Since $K$ is non-trivial and centralizers of non-trivial elements are finite, $C_G(K)$ is finite, which contradicts the fact that $G$ is infinite. We must then have that $K=\{e_G\}$, and so $\eta=\delta_{\{e_G\}}$. The conclusion of the second statement in our theorem follows.
	\end{proof}

	\section{A bound on the max cost of a group of bounded exponent}\label{section: max cost}
	We recall some definitions for this section.
	
	The \emph{maximal cost} of a countable group $G$, denoted by $\mathrm{max\text{-}cost}(G)$, is the supremum of the costs of its essentially free p.m.p.\ actions on standard probability spaces. For finitely generated groups, the Abért--Weiss theorem \cite[Theorem 1]{AbertWeiss2013}, together with Kechris's monotonicity theorem for cost under weak containment \cite[Corollary 10.14]{Kechris2010}, implies that this supremum is attained by every essentially free Bernoulli action; see \cite[Corollary 2]{AbertWeiss2013}. In particular, it is attained by the Bernoulli shift $G\curvearrowright([0,1]^G,\lambda^G)$, where $\lambda$ denotes Lebesgue probability measure. The corresponding characterization in terms of factor of IID random graphs is
	\begin{equation}\label{eq:max cost fiid}
		\mathrm{max\text{-}cost}(G)
		=\inf\left\{ \frac{1}{2} \int_{\mathcal{S}(G)} \deg_{\mathcal{G}}(e_G)\ d\mu(\mathcal{G}) : \mu\in\mathcal{M}_{\mathrm{fiid}}(G,\mathcal{S}(G))
		\right\};
	\end{equation}
	see \cite[Equation (1.2)]{hutchcroftpete2020}. Here $\mathcal{M}_{\mathrm{fiid}}(G,\mathcal{S}(G))$ denotes the set of $G$-invariant Borel probability measures on $\mathcal{S}(G)$ that arise as factors of IID vertex labels. More precisely, $\mu\in \mathcal{M}_{\mathrm{fiid}}(G,\mathcal{S}(G))$ if there exists a $G$-equivariant Borel map $\Phi:[0,1]^G\to\mathcal{S}(G)$ such that $\mu=\Phi_*(\lambda^G)$.
	
	\begin{thm}\label{thm: max cost}
		Let $p \ge 2$ be a prime number, and let $G$ be a finitely generated group of exponent $p$. Denote by $d(G)$ the size of a finite generating set of $G$ of minimal size. Then 
		\[
		\mathrm{max\text{-}cost}(G)\le d(G)\left(1-\frac{1}{p}\right).
		\]
	\end{thm}
	\begin{proof}
		Let us denote $d=d(G)$ and let us choose a generating set $\{s_1,\ldots,s_d\}$ of $G$ of cardinality $d$. Let us write $H_i=\langle s_i\rangle$ for each $i=1,\ldots, d$, and note that, since $p$ is prime, the subgroup $H_i$ has order $p$.
		
		Consider the probability space $([0,1]^G,\lambda^G)$, where $\lambda$ is the Lebesgue probability measure. Note that, since $\lambda$ is non-atomic, given any non-empty finite subset $F\subseteq G$ for $\lambda^G$-almost every $\omega\in [0,1]^G$ it holds that the set of values $\{\omega_f: f\in F\}$ achieves its minimum at a unique $f\in F$. 
        Indeed, given such an $F\subseteq G$, for all $k\geq 1$ we can partition $[0,1]$ into intervals $I_1,\dots ,I_k$ of length $1/k$, and one has 
        \begin{align*}
            \lambda^G\left(\bigcup_{\substack{s,t\in F\\s\neq t}}\{\omega\in [0,1]^G:\omega_s=\omega_t\}\right)&\leq\sum_{\substack{s,t\in F\\s\neq t}}\sum_{i=1}^{k}\lambda^G(\{\omega\in [0,1]^G:\omega_s,\omega_t\in I_i\})\leq \frac{|F|^2}{k} \xrightarrow[k\to \infty]{} 0.
        \end{align*}

        In what follows, we consider the Borel set $\Omega\subseteq [0,1]^G$ of $\lambda^G$-measure $1$ where for every $1\le i\le d$ and for each left coset $F$ of $H_i$, the set of values $\{\omega_f: f\in F\}$ achieves its minimum at a unique $f\in F$.

		For each $1\le i\le d$, let us define the random variable $\Psi_i:\Omega\to \mathrm{Graph}(G)$ as follows. For each left coset $F$ of $H_i$ and every $\omega\in \Omega$ choose the unique $f^*\in F$ such that $\omega_{f^{*}}=\min\{\omega_f: f\in F\}$. Then add the undirected edges $\{f^*,f\}$, $f\in F\backslash \{f^*\}$, to the graph $\Psi_i(\omega)$. Finally, we define $\Psi:\Omega\to \mathrm{Graph}(G)$ as follows: for each $\omega\in \Omega$, the graph $\Psi(\omega)$ has as edge set the union $\bigcup_{1\le i\le d}E(\Psi_i(\omega))$.
		
		It is straightforward to check that $\Psi$ is $G$-equivariant and that for $\lambda^G$-almost every $\omega\in \Omega$ the graph $\Psi(\omega)$ is connected. Hence, we have $\mu\coloneqq \Psi_{*}(\lambda^G)\in \mathcal{M}_{\mathrm{fiid}}(G,\mathcal{S}(G))$.

		Let us define, for each $1\le i \le d$, the event
		\[
		\mathcal{E}_i\coloneqq \left\{\omega\in \Omega: \omega_{e_G}=\min\{\omega_f: f\in H_i\} \right\}.
		\]
		Note that $\lambda^G(\mathcal{E}_i)=\frac{1}{p}$ for each $1\le i \le d$. Then, we have
		\begin{align*}
			\frac{1}{2}\int_{\mathcal{S}(G)}\deg_{\mathcal{G}}(e_G)\ d\mu(\mathcal{G})&=\frac{1}{2}\int_{\Omega}\deg_{\Psi(\omega)}(e_G)\ d\lambda^G(\omega)\\
			&\le \frac{1}{2}\sum_{i=1}^d\int_{\Omega}\deg_{\Psi_i(\omega)}(e_G)\ d\lambda^G(\omega)\\
			&=\frac{1}{2}\sum_{i=1}^d\left((p-1)\lambda^G(\mathcal{E}_i)+1-\lambda^G(\mathcal{E}_i)\right)\\
			&=\frac{d}{2}\left(\frac{p-1}{p}+1-\frac{1}{p}\right)=d\left(1-\frac{1}{p}\right).
		\end{align*}
		
		The characterization of the maximal cost provided by Equation \eqref{eq:max cost fiid} then shows the inequality claimed in our theorem.
	\end{proof}
	\bibliographystyle{alpha}
	\bibliography{biblio}
\end{document}